\UseRawInputEncoding
\documentclass[12pt]{article}
\usepackage[centertags]{amsmath}
\usepackage{amsfonts}
\usepackage{amssymb}
\usepackage{latexsym}
\usepackage{amsthm}
\usepackage{newlfont}
\usepackage{graphicx}
\usepackage{listings}
\usepackage{booktabs}
\usepackage{abstract}
\date{}

\newlength{\defbaselineskip}
\newcommand{\setlinespacing}[1]%
           {\setlength{\baselineskip}{#1 \defbaselineskip}}

\newcommand{\N}{{\mathbb{N}}}

\newcommand{\actaqed}{\hfill $\actabox$}
{\medskip\noindent \textit{Proof of #1. }}%
{\actaqed \medskip}

\def\D{{\mathcal D}}

\def\cA{{\mathcal A}}

\def\cC{{\mathcal C}}

\def \Tr{\mathcal T}
\def \cK{\mathcal K}

\def \cV{\mathcal V}
\def \cU{\mathcal U}

\def \cF{\mathcal F}
\def \cS{\mathcal S}
\def \cX{\mathcal X}

\def \cL{\mathcal L}

\def \cR{\mathcal R}
\def\R{{\mathbb R}}
\def\Z{\mathbb Z}

\def \T{\mathbb T}

\def\bbC{\mathbb C}
\def \<{\langle}
\def\>{\rangle}

\def \ff{\varphi}

\def \sign{\operatorname{sign}}

\def\ba{\mathbf a}

\def\bj{\mathbf j}
\def\bh{\mathbf h}
\def\bx{\mathbf x}
\def\by{\mathbf y}

\def\bk{\mathbf k}

\def\bw{\mathbf w}

\def\bn{\mathbf n}
\def\bs{\mathbf s}

\def\bN{\mathbf N}

\def\De{{\Delta}}

\newtheorem{Theorem}{Theorem}[section]
\newtheorem{Lemma}{Lemma}[section]
\newtheorem{Definition}{Definition}[section]
\newtheorem{Proposition}{Proposition}[section]
\newtheorem{Remark}{Remark}[section]

\numberwithin{equation}{section}

\newcommand{\be}{\begin{equation}}
\newcommand{\ee}{\end{equation}}

\begin{document}

\title{On universal sampling representation}
\author{V.N. Temlyakov\thanks{University of South Carolina, Steklov Institute of Mathematics, Lomonosov Moscow State University, and Moscow Center for Fundamental and Applied Mathematics.  }}
\maketitle
\begin{abstract}
{For the multivariate trigonometric polynomials we study convolution with the corresponding  the de la Vall\'ee Poussin kernel from the point of view of discretization. In other words, we replace the normalized Lebesgue measure by a discrete measure in such a way, which preserves the convolution properties and provides sampling discretization of integral norms. We prove that in the two-variate case the Fibonacci point sets provide an ideal (in the sense of order) solution. We also show that the Korobov point sets provide a suboptimal (up to logarithmic factors) solution for an arbitrary number of variables. }
\end{abstract}

\section{Introduction} 
\label{I} 

{\bf General remarks.} We study universal discretization of convolution of trigonometric polynomials from a given collection of finite-dimensional subspaces.   Let $Q$ be a finite subset of $\Z^d$. Denote
$$
\Tr(Q):= \{f: f=\sum_{\bk\in Q}c_\bk e^{i(\bk,\bx)}\},\quad \D_Q(\bx) := \sum_{\bk\in Q} e^{i(\bk,\bx)}.
$$
The convolution of two polynomials $f,g \in \Tr(Q)$ is 
$$
(f*g)(\bx) := (2\pi)^{-d} \int_{\T^d} f(\by)g(\bx-\by)d\by,
$$
where $\T^d := [0,2\pi)^d$. We are interested in replacing integration with respect to the Lebesgue measure by a finite sum. In other words, we want to replace the normalized Lebesgue measure on $\T^d$ by a discrete measure $\mu_m$ with support on a finite set of points $\{\xi^\nu\}_{\nu=1}^m$ such that $\mu_m(\xi^\nu)=1/m$, $\nu=1,\dots,m$. We recall some of the well known results in this direction. Denote
$$
x^l := 2\pi l/(2N+1), \qquad l = 0, 1, ..., 2N.
$$
It is well known (see, for instance, \cite{VTbookMA}, p.7) that for any $u,v \in \Tr([-N,N])$ we have
$$
\<u, v\> := (2\pi)^{-1}\int_{\T} u(x) \overline{v(x)} dx =
(2N+1)^{-1}\sum_{l=0}^{2N} u(x^l)\overline{ v(x^l)}.
$$
This implies that in the case $Q=[-N,N]$ convolution can be discretized with $\xi^\nu = x^{\nu-1}$, $\nu=1,\dots, 2N+1$. Similarly, in the multivariate case of $Q=\Pi(\bN):=[-N_1,N_1]\times \cdots \times [-N_d,N_d]$, $N_j \in \N_0:=\N\cup \{0\}$, $j=1,\dots,d$, $\bN=(N_1,\dots,N_d)$ 
we denote
\begin{align*}
P(\mathbf N) := \bigl\{\mathbf n = (n_1 ,\dots,n_d):\, n_j\in \N_0 ,\quad
 0\le n_j\le 2N_j  ,\, j = 1,\dots,d \bigr\},
\end{align*}
and set
$$
\bx^{\mathbf n}:=\left(\frac{2\pi n_1}{2N_1+1},\dots,\frac{2\pi n_d}
{2N_d+1}\right),\qquad \mathbf n\in P(\mathbf N) .
$$
Then we can discretize the multivariate convolution of functions from $\Tr(\Pi(\bN))$ using 
the set of points $\{\bx^\bn\}_{\bn\in P(\bN)}$. Clearly, this set of points heavily depends on
$\bN$. We would like to build a universal discretization of convolution. We formulate it in a special case. For $\bj\in\N^d$
define
$$
R(\bj) := \{\bk \in \Z^d :   |k_i| < j_i, \quad i=1,\dots,d\}.
$$
Consider the collection $\cC'(N,d):= \{\Tr(R(\bj)), j_1\cdots j_d\le N\}$. 
We would like to find a set of points $\{\xi^\nu\}_{\nu=1}^m$ with as small $m$ as possible 
with the following properties. 

{\bf Property A.} For any $\bj$, satisfying $|j_1|\cdots|j_d|\le N$, and any $f,g \in 
\Tr(R(\bj))$ we have
$$
(f*g)(\bx) = \frac{1}{m} \sum_{\nu=1}^m f(\xi^\nu)g(\bx-\xi^\nu).
$$
In particular, for $g(\bx) = \D_{R(\bj)}(\bx)$ we have
$$
f(\bx) = \frac{1}{m} \sum_{\nu=1}^m f(\xi^\nu)\D_{R(\bj)}(\bx-\xi^\nu).
$$

{\bf Property B.} Let $1\le p < \infty$. For any $\bj$, satisfying $j_1\cdots j_d\le N$, and any $f \in \Tr(R(\bj))$ we have
$$
\|f\|_p^p \le C(p,d) \frac{1}{m} \sum_{\nu=1}^m |f(\xi^\nu)|^p.
$$

In other words we want a universal sampling representation with additional good properties (Property B). 

In Section \ref{A} we prove that in the case $d=2$ the Fibonacci point set $\cF_n$ (with an appropriate $n$) provides Properties 
A and B with the best possible (in the sense of order) bound on $m$: $m\le CN$. 

In Section \ref{B} we study properties of the following operator 
$$
V_{Q_r,b_n}(\ba)(\bx):= \frac{1}{b_n} \sum_{\nu=1}^{b_n} a_\nu \cV_{Q_r}(\bx-\by^\nu),
$$
where $Q_r$ is the step hyperbolic cross and $\cV_{Q_r}$ is the de la Vall\'ee Poussin kernel
for it (see Section \ref{B}). 

In Section \ref{C} we extend results of Section \ref{A} to the case $d\ge 3$. Instead of the Fibonacci point sets we consider the Korobov point sets. We obtain results somewhat similar to those from Section \ref{A} but not as sharp as results on the Fibonacci point sets. We show that the Korobov point sets provide suboptimal (up to logarithmic factors) results for an arbitrary $d$.

In Section \ref{D} we present a discussion of known results on universal sampling discretization and their relations with our new results. Also, we formulate some open problems. 

{\bf Main results.} We now formulate the main results of the paper. Let $\{b_n\}_{n=0}^{\infty}$, $b_0=b_1 =1$, $b_n = b_{n-1}+b_{n-2}$,
$n\ge 2$, -- be the Fibonacci numbers.

Denote 
$$
\mathbf y^{\mu}:=\bigl(2\pi\mu/b_n, 2\pi\{\mu
b_{n-1}/b_n\}\bigr), \quad \mu = 1,\dots,b_n,\quad \cF_n:=\{\by^\mu\}_{\mu=1}^{b_n}.
$$
In this definition $\{a\}$ is the fractional part of the number $a$. The cardinality of the set $\cF_n$ is equal to $b_n$.  Let $\ell^m_{p,m}$ be $\R^m$ equipped with the norm
$$
\|\bx\|_{p,m} := \left(\frac{1}{m}\sum_{i=1}^m |x_i|^p\right)^{1/p},\quad 1\le p<\infty; \quad \|\bx\|_\infty := \max_{i}|x_i|.
$$
In Section \ref{A} we prove the following result.

\begin{Theorem}\label{IT1} Let $\gamma$ be from Lemma \ref{AL1} and let $1\le p\le\infty$. The Fibonacci point set $\cF_n$ provides the following two properties for the collection $\cC'(N,2)$ with $N=[\gamma b_n/4]$.  

{\bf (I).} For any $\bj$, satisfying $4j_1j_2\le \gamma b_n$, and any $f,g \in 
\Tr(R(\bj))$ we have
$$
(f*g)(\bx) = \frac{1}{b_n} \sum_{\nu=1}^{b_n} f(\by^\nu)g(\bx-\by^\nu).
$$

{\bf (II).} Let $\cV_\bj(\bx)$ be the de la Vall\'ee Poussin kernels for $R(\bj)$ (see Section \ref{A} for the definition).
Then for any $\bj$ satisfying $4j_1j_2\le \gamma b_n$ and any $f\in \Tr(R(\bj))$ we have
$$
\left\|\frac{1}{b_n}\sum_{\nu=1}^{b_n}  a_\nu\cV_\bj(\bx-\by^\nu)\right\|_p \le 9 \|\ba\|_{p,b_n},\quad \ba = (a_1,\dots,a_{b_n}).
$$
\end{Theorem}

In Section \ref{C} we extend Theorem \ref{IT1} to the case of $d\ge 3$ by considering the Korobov point sets instead of the Fibonacci point sets. Results of Section \ref{C} are not as 
sharp as results of Section \ref{A} -- in the bounds on the number of points $m$ we have an extra logarithmic factor.

Let $\cV_{Q_r}(\bx) $ be the hyperbolic cross de la Vall\'ee Poussin kernel (see the definition in Section \ref{B}).
We are interested in studying the operator $V_{Q_r,b_n}: \ell^{b_n}_{p,b_n} \to L_p(\T^2)$ defined as
$$
 V_{Q_r,b_n}(\ba)(\bx):= \frac{1}{b_n} \sum_{\nu=1}^{b_n} a_\nu \cV_{Q_r}(\bx-\by^\nu),\quad \ba=(a_1,\dots,a_{b_n}).
$$
In Section \ref{B} we prove the following result.

\begin{Theorem}\label{IT2} Let $r\in \N$ be such that $2^r \le \gamma b_n$. Then we have for \newline $1\le p\le\infty$
\be\label{I1}
\|V_{Q_r,b_n}\|_{\ell^{b_n}_{p,b_n} \to L_p} \le C r^{\theta(p)},\quad \theta(p):= \max(1/p,1-1/p), 
\ee
with an absolute constant $C$.
\end{Theorem}

Note that Theorem \ref{IT2} is sharp in the cases $1\le p\le 2$ and $p=\infty$. Indeed, let us take $\ba = (1,0,\dots,0)$. Then for $1\le p\le 2$ Theorem \ref{IT2} gives
 \be\label{I2}
\|b_n^{-1}\cV_{Q_r}(\bx-\by^1)\|_p \le Cb_n^{-1/p}r^{1/p}.
\ee
Let $r$ be such that $2^r \asymp b_n$ ($2^r$ is of order $b_n$). It is known (see, for instance, \cite{VTbookMA}, p.140) that
 \be\label{I3}
\|\cV_{Q_r}\|_p \ge c(p) 2^{(1-1/p)r}r^{1/p} \asymp b_n^{1-1/p} (\log b_n)^{1/p}.
\ee
Comparing the above two inequalities, we see that the upper bound in (\ref{I2}) provided by 
Theorem \ref{IT2} coincides (in the sense of order) with the lower bound in (\ref{I3}). 
The lower bound in the case $p=\infty$ is straight forward.

 \section{The Fibonacci point sets}
\label{A}  
 
For the continuous functions of two
variables, which are $2\pi$-periodic in each variable, we define
cubature formulas
$$
\Phi_n(f) :=b_n^{-1}\sum_{\mu=1}^{b_n}f\bigl(2\pi\mu/b_n,
2\pi\{\mu b_{n-1} /b_n \}\bigr),
$$
 called the {\it Fibonacci cubature formulas}.  Denote
 $$
\Phi(\mathbf k) := b_n^{-1}\sum_{\mu=1}^{b_n}e^{i(\mathbf k,\mathbf y^{\mu})}.
$$
Then
\be\label{A1}
\Phi_n (f) =\sum_{\mathbf k}\hat f(\mathbf k)\Phi(\mathbf k),\quad \hat f(\bk) := (2\pi)^{-2}\int_{\T^2} f(\bx)e^{-i(\bk,\bx)}d\bx,
\ee
where for the sake of simplicity we may assume that $f$ is a
trigonometric polynomial. It is clear that (\ref{A1}) holds for $f$ with absolutely convergent Fourier series. 

It is easy to see that the following relation holds
$$
\Phi(\mathbf k)=
\begin{cases}
1&\quad\text{ for }\quad \mathbf k\in L(n)\\
0&\quad\text{ for }\quad \mathbf k\notin L(n),
\end{cases}
$$
where
$$
L(n) :=\bigl\{ \mathbf k = (k_1,k_2):k_1 + b_{n-1} k_2\equiv 0
\qquad \pmod {b_n}\bigr\}.
$$
Denote $L(n)':= L(n)\backslash\{\mathbf 0\}$. For $N\in\N$ define the {\it hyperbolic cross} in dimension $2$ as follows:
$$
\Gamma(N):=\Gamma(N,2):= \left\{\bk\in\Z^2: \prod_{j=1}^2 \max(|k_j|,1) \le N\right\}.
$$
The following lemma is well known (see, for instance, \cite{VTbookMA}, p.274).

\begin{Lemma}\label{AL1} There exists an absolute constant $\gamma > 0$
such that for any $n > 2$ for the $2$-dimensional hyperbolic cross we have
$$
\Gamma(\gamma b_n)\cap \bigl(L(n)\backslash\{\mathbf 0\}\bigr) =
\varnothing
$$
and, therefore, for any $f\in \Tr(\Gamma(N))$ with $N\le \gamma b_n$ we have
$$
\Phi_n(f) = (2\pi)^{-2}\int_{\T^2} f(\bx)d\bx.
$$
\end{Lemma}

{\bf Proof of Theorem \ref{IT1}.} 
  We begin with a proof of part (I) of Theorem \ref{IT1}. It is a direct corollary of Lemma \ref{AL1}. For $f,g\in\Tr(\bj)$ we have for each $\bx\in\T^2$ that $f(\by)g(\bx-\by)\in \Tr(2\bj)$ and, taking into account our assumption $4j_1j_2\le \gamma b_n$, by Lemma \ref{AL1} we obtain
 $$
(f*g)(\bx) = (2\pi)^{-2}\int_{\T^2} f(\by)g(\bx-\by)d\by = \Phi_n(f(\cdot)g(\bx-\cdot) )
$$
$$
= \frac{1}{b_n}\sum_{\nu=1}^m  f(\by^\nu)g(\bx-\by^\nu).
$$

We now proceed to the proof of  part (II) of Theorem \ref{IT1}. 
   We need some classical trigonometric polynomials for our further argument (see \cite{Z} and \cite{VTbookMA}). We begin with the univariate case. 
 The Dirichlet kernel of order $j$:
$$
\mathcal D_j (x):= \sum_{|k|\le j}e^{ikx} = e^{-ijx} (e^{i(2j+1)x} - 1)
(e^{ix} - 1)^{-1} 
$$
$$
=\bigl(\sin (j + 1/2)x\bigr)\bigm/\sin (x/2)
$$
   is an even trigonometric polynomial.  The Fej\'er kernel of order $j - 1$:
$$
\mathcal K_{j} (x) := j^{-1}\sum_{k=0}^{j-1} \mathcal D_k (x) =
\sum_{|k|\le n} \bigl(1 - |k|/j\bigr) e^{ikx} 
$$
$$
=\bigl(\sin (jx/2)\bigr)^2\bigm /\bigl(j (\sin (x/2)\bigr)^2\bigr).
$$
The Fej\'er kernel is an even nonnegative trigonometric
polynomial in $\Tr(j-1)$.  It satisfies the obvious relations
\be\label{FKm}
\| \mathcal K_{j} \|_1 = 1, \qquad \| \mathcal K_{j} \|_{\infty} = j.
\ee
The de la Vall\'ee Poussin kernel
\be\label{A2}
\mathcal V_{j} (x) := j^{-1}\sum_{l=j}^{2j-1} \mathcal D_l (x)= 2\cK_{2j}(x)-\cK_j(x) 
\ee
is an even trigonometric
polynomial of order $2j - 1$.

In the two-variate case  define the Fej\'er and de la Vall\'ee Poussin kernels as follows:
$$
\cK_\bj(\bx):=  \cK_{j_1}(x_1) \cK_{j_2}(x_2), \qquad \mathcal V_{\mathbf j} (\bx) := \cV_{j_1}(x_1) \cV_{j_2}(x_2)    ,\qquad
\mathbf j = (j_1,j_2) .
$$

The statement of the part (II) of Theorem \ref{IT1} follows from Lemma \ref{AL2}.
\begin{Lemma}\label{AL2} Let  $1\le p\le \infty$. Then for any $\bj$ satisfying $4j_1j_2\le \gamma b_n$ and any $f\in \Tr(R(\bj))$ we have
$$
\left\|\frac{1}{b_n}\sum_{\nu=1}^m  a_\nu\cV_\bj(\bx-\by^\nu)\right\|_p \le 9 \|\ba\|_{p,b_n}.
$$
\end{Lemma}
\begin{proof} Define the operator $V_\bj : \ell^{b_n}_{p,b_n} \to L_p$ as follows
$$
V_\bj(\ba) :=\frac{1}{b_n}\sum_{\nu=1}^{b_n}  a_\nu\cV_\bj(\bx-\by^\nu),\quad \ba=(a_1,\dots,a_{b_n}).
$$
We treat two extreme cases $p=1$ and $p=\infty$ and then use the classical Riesz-Thorin interpolation theorem. 

{\bf Case $p=1$.} Using the well known fact, which follows directly from (\ref{A2}) and (\ref{FKm}), that for the univariate de la Vall\'ee Poussin kernels we have the bound $\|\cV_j\|_1 \le 3$, we obtain
$$
\left\|\frac{1}{b_n}\sum_{\nu=1}^m  a_\nu\cV_\bj(\bx-\by^\nu)\right\|_1 \le 9\|\ba\|_{1,b_n},
$$
which means that 
$$
\|V_\bj\|_{\ell^{b_n}_{1,b_n} \to L_1} \le 9.
$$

{\bf Case $p=\infty$.} For any $\bx$ we have
\be\label{A3}
\left|\frac{1}{b_n}\sum_{\nu=1}^{b_n}  a_\nu\cV_\bj(\bx-\by^\nu)\right| \le \|\ba\|_\infty \frac{1}{b_n}\sum_{\nu=1}^{b_n}  |\cV_\bj(\bx-\by^\nu)|.
\ee
Further,
\be\label{A4}
|\cV_\bj(\by)| \le (2\cK_{2j_1}(y_1)+\cK_{j_1}(y_1))(2\cK_{2j_2}(y_2)+\cK_{j_2}(y_2)).
\ee
Using our assumption $4j_1j_2\le \gamma b_n$ and Lemma \ref{AL1} we obtain from (\ref{A3}),  (\ref{A4}), and (\ref{FKm}) that
$$
 \|V_\bj\|_{\ell^{\infty_n}_{1,b_n} \to L_\infty} \le 9.
$$
 
 It remains to use the Riesz-Thorin interpolation theorem and complete the proof.
 \end{proof}

\section{Fibonacci points and hyperbolic cross polynomials}
\label{B}

Consider the following special univariate trigonometric polynomials. Let $s$ be a nonnegative integer. Define
$$
\cA_0(x) := 1,\quad \cA_1(x):= \cV_1(x)-1,\quad \cA_s(x):= \cV_{2^{s-1}}(x) - \cV_{2^{s-2}}(x), \quad s\ge 2,
$$
where $\cV_m$ are the de la Vall\'ee Poussin kernels defined above. Then we have
\be\label{B1}
\sum_{j=0}^s \cA_j(x) = \cV_{2^{s-1}}(x),\quad \cV_{2^{-1}}(x) := 1.
\ee

In the multivariate case 
$\bx=(x_1,\dots,x_d)$ and $\bs=(s_1,\dots,s_d)\in \Z_+^d$ define
$$
\cA_\bs(\bx):= \cA_{s_1}(x_1)\cdots \cA_{s_d}(x_d).
$$
For $\bs\in\Z^d_+$
define
$$
\rho (\bs) := \{\bk \in \Z^d : [2^{s_j-1}] \le |k_j| < 2^{s_j}, \quad j=1,\dots,d\}
$$
where $[x]$ denotes the integer part of $x$. We define the step hyperbolic cross
$Q_r$ as follows
$$
Q_r := \cup_{\bs:\|\bs\|_1\le r} \rho(\bs)
$$
and the corresponding set of the hyperbolic cross polynomials as 
$$
\Tr(Q_r) := \{f: f=\sum_{\bk\in Q_r} c_\bk e^{i(\bk,\bx)}\}.
$$
We define the hyperbolic cross de la Vall\'ee Poussin kernels as follows
$$
\cV_{Q_r}(\bx) := \sum_{\|\bs\|_1\le r}\cA_\bs(\bx).
$$
Using notations $\bx^d:=(x_1,\dots,x_{d-1})$, $\bs^d:=(s_1,\dots,s_{d-1})$, and (\ref{B1}), we  rewrite this definition as 
\be\label{B2'}
 \cV_{Q_r}(\bx) := \sum_{\|\bs^d\|_1\le r}\cA_{\bs^d}(\bx^d)\sum_{s_d=0}^{r-\|\bs^d\|_1} \cA_{s_d}(x_d) = \sum_{\|\bs^d\|_1\le r}\cA_{\bs^d}(\bx^d)\cV_{2^{r-\|\bs^d\|_1-1}}(x_d).
 \ee
 In particular, this implies that
 \be\label{B2}
 \|\cV_{Q_r}\|_1 \le C(d)r^{d-1}.
 \ee
 
 We are interested in studying the operator $V_{Q_r,b_n}: \ell^{b_n}_{p,b_n} \to L_p(\T^2)$ defined as
 $$
 V_{Q_r,b_n}(\ba)(\bx):= \frac{1}{b_n} \sum_{\nu=1}^{b_n} a_\nu \cV_{Q_r}(\bx-\by^\nu).
 $$
 Relation (\ref{B2}) implies 
 \be\label{B4}
 \|V_{Q_r,b_n}\|_{\ell^{b_n}_{1,b_n} \to L_1} \le C_1r.
 \ee
 It is clear that 
 \be\label{B5}
 \|V_{Q_r,b_n}\|_{\ell^{b_n}_{\infty,b_n} \to L_\infty} \le \max_\bx\frac{1}{b_n} \sum_{\nu=1}^{b_n} | \cV_{Q_r}(\bx-\by^\nu)|.
 \ee
 \begin{Lemma}\label{BL2} Let $r\in \N$ be such that $2^r \le \gamma b_n$. Then we have
$$
\max_\bx\frac{1}{b_n} \sum_{\nu=1}^{b_n} | \cV_{Q_r}(\bx-\by^\nu)| \le C_\infty r.
$$
\end{Lemma}
\begin{proof} Represent
\be\label{B7}
\cV_{N_j}(t) = 2\cK_{2N_j}(t)-\cK_{N_j}(t).
\ee
 By representation (\ref{B2'}) we write
\be\label{B8}
\frac{1}{b_n} \sum_{\nu=1}^{b_n} | \cV_{Q_r}(\bx-\by^\nu)| \le \sum_{s_1=0}^{r} \frac{1}{b_n} \sum_{\nu=1}^{b_n}|\cA_{s_1}(x_1-y^\nu_1)\cV_{2^{r-s_1-1}}(x_2-y^\nu_2)|.
\ee
Using the fact that the Fej{\'e}r kernel is a nonnegative polynomial and applying Lemma \ref{AL1},  we obtain from (\ref{B8})
$$
\frac{1}{b_n} \sum_{\nu=1}^{b_n} | \cV_{Q_r}(\bx-\by^\nu)| \le C_\infty r.
$$

\end{proof}

Lemma \ref{BL2} and inequalities (\ref{B4}) and (\ref{B5}) imply by the Riesz-Thorin interpolation theorem the following proposition.

\begin{Proposition}\label{BP1} Let $1\le p\le\infty$ and let $r\in \N$ be such that $2^r \le \gamma b_n$. Then we have
\be\label{B9}
 \|V_{Q_r,b_n}\|_{\ell^{b_n}_{p,b_n} \to L_p} \le C_p r,\quad C_p=C_1^{1/p} C_\infty^{1-1/p}.
 \ee
\end{Proposition}

The following function, which we call the Fibonacci Sum of $\cV_{Q_r}$,
$$
\cF\cS\cV_{Q_r}(\bx):= \frac{1}{b_n} \sum_{\nu=1}^{b_n} | \cV_{Q_r}(\bx-\by^\nu)|
$$
played an important role in the proof of Proposition \ref{BP1}. We note the following interesting property of this function.
\begin{Proposition}\label{BP2} Let $r\in \N$ be such that $2^r \le \gamma b_n$. Then we have for all $1\le p\le \infty$ that
$$
C'r \le \|\cF\cS\cV_{Q_r}\|_p \le C''r
$$
with two positive absolute constants $C'$ and $C''$.
\end{Proposition}
\begin{proof} The upper bound follows from Lemma \ref{BL2}. The lower bound follows from the known fact (see, for instance, \cite{VTbookMA}, p.140, Lemma 4.2.3)
$$
\|\cV_{Q_r}\|_1 \ge C' r.
$$
\end{proof}

We now show how inequality (\ref{B9}) can be improved in the case $1<p<\infty$. We begin with the case $p=2$ and a simple well known Lemma \ref{BL3}. 

\begin{Lemma}\label{BL3} Let a system $\cU_m=\{u_i\}_{i=1}^m$ satisfy the conditions
\be\label{B10}
\|u_i\|_2=1,\quad \sum_{j=1}^m |\<u_i,u_j\>| \le C_0,\quad i=1,\dots,m.
\ee
Then for any $\ba=(a_1,\dots,a_m) \in \bbC^m$ we have
$$
\|\sum_{i=1}^m a_iu_i\|_2 \le C_0^{1/2}\|\ba\|_2.
$$
\end{Lemma}
\begin{proof} Consider the Gramm matrix $U:=[\<u_i,u_j\>]_{i,j=1}^m$. Then
\be\label{B12}
\|\sum_{i=1}^m a_iu_i\|_2^2 = \sum_{i=1}^m a_i \sum_{j=1}^m \<u_i,u_j\>\bar a_j = \<\ba,\bar U \ba\>.
\ee
 Consider $U$ as an operator from $\bbC^m$ to $\bbC^m$. Then
$$
 \|U\|_{\ell^m_\infty \to \ell^m_\infty} \le \max_i \sum_{j=1}^m |\<u_i,u_j\>| \le C_0,
$$
 and
$$
 \|U\|_{\ell^m_1 \to \ell^m_1} \le \max_j \sum_{i=1}^m |\<u_i,u_j\>| \le C_0.
$$
Therefore, by the Riesz-Thorin interpolation theorem we obtain
\be\label{B15}
  \|\bar U\|_{\ell^m_2 \to \ell^m_2}=\|U\|_{\ell^m_2 \to \ell^m_2} \le C_0.
 \ee
 By (\ref{B12}) and (\ref{B15}) we conclude
 $$
 \|\sum_{i=1}^m a_iu_i\|_2^2 \le \|\ba\|_2\|\bar U \ba\|_2 \le  \|\bar U\|_{\ell^m_2 \to \ell^m_2}\|\ba\|_2^2 \le C_0\|\ba\|_2^2,
 $$
 which completes the proof of Lemma \ref{BL3}.

\end{proof}

{\bf Proof of Theorem \ref{IT2}.} Bound (\ref{I1}) with $p=\infty$ follows directly from (\ref{B5}) and Lemma \ref{BL2}. The case $p=1$ in (\ref{I1}) is covered by (\ref{B4}). Let us consider the case $p=2$. We now apply Lemma \ref{BL3} in the case $u_i(\bx)=\cV_{Q_r}(\bx-\by^i)/\|\cV_{Q_r}\|_2$. Then
$$
\<u_i,u_j\> = \cV_{Q_r}^*(\by^j-\by^i)\|\cV_{Q_r}\|_2^{-2},
$$
where for $f\in L_2(\T^2)$ we denote
$$
f^*(\bx):= (f\ast f)(\bx) := (2\pi)^{-2}\int_{\T^2} f(\bx-\by)f(\by)d\by.
$$
In order to apply Lemma \ref{BL3} we need to satisfy condition (\ref{B10}).

 \begin{Lemma}\label{BL4} Let $r\in \N$ be such that $2^r \le \gamma b_n$. Then we have
$$
\max_\bx\frac{1}{b_n} \sum_{\nu=1}^{b_n} | \cV_{Q_r}^*(\bx-\by^\nu)| \le C^*_\infty r.
$$
\end{Lemma}
\begin{proof} The argument follows the same lines as the proof of Lemma \ref{BL2}. We use representations (\ref{B2'}) with $d=2$ and (\ref{B7}), the fact that if $f$ is nonnegative then $f^*$ is also nonnegative, and the fact that 
$$
\cA_s \ast \cA_{s'} =0 \quad \text{provided}\quad |s-s'|\ge 2.
$$
\end{proof}
Then Lemma \ref{BL3} and Lemma \ref{BL4} imply
$$
\left\|\frac{1}{b_n} \sum_{\nu=1}^{b_n} a_\nu \cV_{Q_r}(\bx-\by^\nu)\right\|_2=\left\|\sum_{\nu=1}^{b_n}\frac{a_\nu \|\cV_{Q_r}\|_2}{b_n} u_\nu(\bx)\right\|_2
$$
$$
  \le \|\cV_{Q_r}\|_2^{-1}(b_nC_\infty^*r)^{1/2}\|\ba\|_2 \|\cV_{Q_r}\|_2/b_n = (C_\infty^*r)^{1/2}\|\ba\|_{2,b_n}.
$$
Thus, we have proved (\ref{I1}) in three cases $p=\infty$, $p=1$, and $p=2$. Applying the Riesz-Thorin interpolation theorem two times, for the pairs $(1,2)$ and $(2,\infty)$, we obtain (\ref{I1}) for all $1\le p\le \infty$ and complete the proof.

Consider now the following kernel, which is closely related to $\cV_{Q_r}$
$$
\Delta\cV_{Q_r} := \cV_{Q_r} - \cV_{Q_{r-1}} = \sum_{\|bs\|_1=r} \cA_\bs.
$$
 It turns out that we can use other technique and prove a better bound in the case $2<p<\infty$ than (\ref{I1}).

 \begin{Theorem}\label{BT2} Let $r\in \N$ be such that $2^r \le \gamma b_n/4$. Then  
  for $1\le p<\infty$ we have
$$
\|\De V_{Q_r,b_n}\|_{\ell^{b_n}_{p,b_n} \to L_p} \le C(p) r^{\max(1/p,1/2)}.
$$
\end{Theorem}
\begin{proof} Rewrite
$$
f(\bx):= \frac{1}{b_n} \sum_{\nu=1}^{b_n} a_\nu \De \cV_{Q_r}(\bx-\by^\nu)   = \sum_{\|\bs\|_1= r}\frac{1}{b_n} \sum_{\nu=1}^{b_n} a_\nu \cA_{\bs}(\bx-\by^\nu).
$$
It is well known that the Littlewood-Paley theorem implies the following inequality (see, for instance \cite{VTbookMA}, p.513) for $2<p<\infty$
\be\label{B21}
\|f\|_p^2 \le C(p) \sum_{\|\bs\|_1= r}\left\|\frac{1}{b_n} \sum_{\nu=1}^{b_n} a_\nu \cA_{\bs}(\bx-\by^\nu)\right\|_p^2,
\ee
and for $1< p\le 2$
 \be\label{B22}
\|f\|_p^p \le C(p) \sum_{\|\bs\|_1= r}\left\|\frac{1}{b_n} \sum_{\nu=1}^{b_n} a_\nu \cA_{\bs}(\bx-\by^\nu)\right\|_p^p.
\ee
In case $p=1$ inequality (\ref{B22}) obviously holds as well. 

Applying an analog of Lemma \ref{AL2}, we obtain
 \be\label{B23}
\left\|\frac{1}{b_n} \sum_{\nu=1}^{b_n} a_\nu \cA_{\bs}(\bx-\by^\nu)\right\|_p \le C \|\ba\|_{p,b_n}.
\ee
It remains to substitute (\ref{B23}) in (\ref{B21}) and in (\ref{B22}) and complete the proof.

\end{proof}

\section{The Korobov point sets}
\label{C}

In this section we extend results of Section \ref{A} to the case $d\ge 3$. Instead of the Fibonacci point sets we consider the Korobov point sets. We obtain results somewhat similar to those from Section \ref{A} but not as sharp as results on the Fibonacci point sets. It is a well known phenomenon in numerical integration.  Here we study the Korobov cubature formulas instead of the 
 Fibonacci cubature formulas. We prove a conditional result under the assumption that the Korobov cubature formulas are exact on a certain subspace of trigonometric polynomials with frequencies  from a hyperbolic cross. 
 There are results that guarantee existence of such cubature formulas.  
  
 Let $m\in\N$, $\mathbf h := (h_1,\dots,h_d)$, $h_1,\dots,h_d\in\Z$.
We consider the cubature formulas
$$
P_m (f,\mathbf h):= m^{-1}\sum_{\mu=1}^{m}f\left ( 2\pi\left  \{\frac{\mu h_1}
{m}\right\},\dots, 2\pi\left \{\frac{\mu h_d}{m}\right\}\right),
$$
which are called the {\it Korobov cubature formulas}.  In the case $d=2$, $m=b_n$, $\mathbf h = (1,b_{n-1})$ we have
$$
P_m (f,\mathbf h) =\Phi_n (f):= \frac{1}{b_n}\sum_{\by\in \cF_n} f(\by).
$$

Denote 
$$
\mathbf w^{\mu}:=\left ( 2\pi\left \{\frac{\mu h_1}
{m}\right\},\dots, 2\pi\left \{\frac{\mu h_d}{m}\right\}\right), \quad \mu = 1,\dots,m,\quad \cR_m(\bh):=\{\bw^\mu\}_{\mu=1}^m.
$$
The set $\cR_m(\bh)$ is called the {\it Korobov point set}. Further, denote
$$
S(\mathbf k, \bh) \,:=\, P_m\left(e^{i(\mathbf k,\mathbf x)},\bh\right)
\,=\, m^{-1}\sum_{\mu=1}^{m}e^{i(\mathbf k,\mathbf w^{\mu})}.
$$
Note that
\be\label{C1}
P_m (f,\bh) =\sum_{\mathbf k}\hat f(\mathbf k)\, S(\mathbf k,\bh),\quad 
\hat f(\bk) := (2\pi)^{-d}\int_{\T^d} f(\bx)\, e^{-i(\bk,\bx)}d\bx,
\ee
where for the sake of simplicity we may assume that $f$ is a
trigonometric polynomial. It is clear that (\ref{C1}) holds for $f$ with absolutely convergent Fourier series. 

It is easy to see that the following relation holds
\be\label{C2}
S(\mathbf k,\bh)=
\begin{cases}
1&\quad\text{ for }\quad \mathbf k\in \cL(m,\bh),\\
0&\quad\text{ for }\quad \mathbf k\notin \cL(m,\bh),
\end{cases}
\ee
where
$$
\cL(m,\mathbf h) := \bigl\{ \mathbf k:(\mathbf h,\mathbf k) \equiv 0 \quad
\pmod{m}\bigr\},\quad \cL(m,\mathbf h)':= \cL(m,\bh)\backslash\{\mathbf 0\} .
$$
  For $N\in\N$ define the {\it hyperbolic cross}   by
$$
\Gamma(N,d):= \left\{\bk= (k_1,\dots,k_d)\in\Z^d\colon \prod_{j=1}^d \max(|k_j|,1) \le N\right\}.
$$
Denote 
$$
\Tr(N,d) := \left\{f\, :\, f(\bx)= \sum_{\bk\in \Gamma(N,d)} c_\bk e^{i(\bk,\bx)}\right\}.
$$
It is easy to see that the condition 
\be\label{ex}
P_m(f,\bh) = \hat f(\mathbf 0), \quad f\in \Tr(N,d),
\ee
is equivalent to the condition
\be\label{exs}
\Gamma(N,d)\cap \cL(m,\bh)'  =
\varnothing.
\ee
\begin{Definition}\label{exact} We say that the Korobov cubature formula $P_m(\cdot,\bh)$ is exact on $\Tr(N,d)$ if condition \eqref{ex} (equivalently, condition \eqref{exs}) is satisfied.
\end{Definition} 

 {\bf Special Korobov point sets.} Let $L\in \N$ be given. Clearly, we are interested in as small $m$ as possible such that there exists a Korobov cubature formula, which is exact on $\Tr(L,d)$. In the case of $d=2$ the Fibonacci cubature formula is an ideal in a certain sense choice. 
There is no known Korobov cubature formulas in case $d\ge 3$, which are as good as  the Fibonacci cubature formula in case $d=2$. We now formulate some known results in this direction. Consider  a special case $\mathbf h = (1,h,h^2,\dots,h^{d-1})$, $h\in\N$. In this case we write in the notation of $\cR_m(\bh)$ and $P_m(\cdot,\bh)$ the scalar $h$ instead of the vector $\mathbf h$, namely, $\cR_m(h,d)$ and $P_m(\cdot,h,d)$. The following Lemma \ref{CL1} is a well known result (see, for instance \cite{VTbookMA}, p.285). 

\begin{Lemma}\label{CL1} Let $m$ and $L$ be a prime  
and a natural number, respectively, such that
\be\label{C3}
\bigl|\Gamma(L,d)\bigr| < (m-1)/d .
\ee
Then there is a natural number $h\in [1,m)$ such that for all
$\mathbf k\in\Gamma(L,d)$, $\mathbf k\ne\mathbf 0$
\be\label{C4}
k_1 + hk_2 +\dots+h^{d-1}k_d\not\equiv 0\qquad \pmod {m}.
\ee
Therefore, for any $f\in \Tr(L,d)$ we have 
$$
P_m(f,h,d) = (2\pi)^{-d}\int_{\T^d} f(\bx)d\bx.
$$
\end{Lemma}

Note that the cardinality of $\Gamma(L,d)$ is of order $L(\log L)^{d-1}$ and, therefore, the largest $L$, satisfying (\ref{C3}), is of order $m(\log m)^{1-d}$. 

In the same way as Theorem \ref{IT1} was derived from Lemma \ref{AL1} in Section \ref{A} 
the following Theorem \ref{CT1} can be derived for special Korobov point sets. We do not present this proof here. 

\begin{Theorem}\label{CT1} Let the Korobov cubature formula $P_m(\cdot,\bh)$ be exact on $\Tr(L,d)$ and let $1\le p\le\infty$. Then  the Korobov point set $\cR_m(\bh)$ provides the following two properties for the collection $\cC'(L,d)$.  

{\bf (I).} For any $\bj\in\N^d$, satisfying $2^d\prod_{i=1}^dj_i\le L$, and any $f,g \in 
\Tr(R(\bj))$ we have
$$
(f*g)(\bx) = \frac{1}{m} \sum_{\nu=1}^{m} f(\bw^\nu)g(\bx-\bw^\nu).
$$

{\bf (II).} Let $\cV_\bj(\bx):= \prod_{i=1}^d \cV_{j_i}(x_i)$ be the de la Vall\'ee Poussin kernels for $R(\bj)$.
Then for any $\bj$ satisfying $2^d\prod_{i=1}^dj_i\le L$ and any $f\in \Tr(R(\bj))$ we have
$$
\left\|\frac{1}{m}\sum_{\nu=1}^{m}  a_\nu\cV_\bj(\bx-\bw^\nu)\right\|_p \le 3^d \|\ba\|_{p,m},\quad \ba = (a_1,\dots,a_{m}).
$$
\end{Theorem}

\begin{Remark}\label{CR1} Lemma \ref{CL1} implies that for any $L\in\N$ there exist $\bh$ 
and $m\le C(d)L(\log L)^{d-1}$ with some positive $C(d)$ such that statements (I) and (II) of Theorem \ref{CT1} hold.
\end{Remark}

\section{Discussion}
\label{D}

The property of a numerical method of approximation or presentation (recovery) to be universal is very important. The classical concept of unsaturated methods (see, for instance, \cite{Ba}) is the universality property with respect to smoothness. Later, universality with respect to 
anisotropy was introduced (see \cite{Tem16}). Study of universality in numerical integration (see \cite{Tem23}, \cite{T11}, and the book \cite{VTbookMA}, section 6.8) and in linear approximation (see \cite{Tem16} and \cite{VTbookMA}, section 5.4) brought new phenomena. Universality concept under the names {\it adaptive learning} and {\it distribution-free theory of regression} 
is very important in learning theory (see \cite{GKKW} and \cite{VTbook}, Chapter 4). 

Recently, because of demand on nonlinear approximation importance of the universality property has increased. We illustrate it on the example of sparse approximation. Suppose we have a finite dictionary $\D_n:=\{g_j\}_{j=1}^n$ of functions from $L_p(\Omega,\mu)$. Applying the strategy of sparse $m$-term approximation with respect to 
$\D_n$ we obtain a collection of all subspaces spanned by at most $m$ elements of $\D_n$ as a possible source of approximating (representing) elements. Therefore, we would like to build a discretization scheme, which works well for all such subspaces. This kind of discretization is called {\it universal discretization}. There are some known results on universal discretization, which show that it is a very interesting and deep area of research (see, for instance, \cite{VT159}, \cite{DPTT}, \cite{DT}).  

 For a more detailed discussion of 
universality in approximation and learning theory we refer the reader to \cite{Tem16}, 
 \cite{T11}, \cite{DTU}, \cite{VT160}, \cite{GKKW},
\cite{BCDDT}, \cite{VT113}.  

\subsection{Discretization} 
It is well known and easy to check that for centrally symmetric $Q$ the problem of discretization 
of the convolution on $\Tr(Q)$: For any $f,g\in\Tr(Q)$
$$
(2\pi)^{-d}\int_{\T^d} f(\by)g(\bx-\by)d\by = \frac{1}{m}\sum_{\nu=1}^m f(\xi^\nu)g(\bx-\xi^\nu)
$$
is equivalent to the problem of exact discretization of the $L_2$ norm on $\Tr(Q)$: For any $f\in \Tr(Q)$
$$
\|f\|_2^2 =  \frac{1}{m}\sum_{\nu=1}^m |f(\xi^\nu)|^2.
$$
Thus, part (I) of Theorem \ref{IT1} is equivalent to the exact discretization of the $L_2$ norm on $\Tr(R(\bj))$. 

We now discuss part (II) of Theorem \ref{IT1}. We prove the following conditional statement.

\begin{Proposition}\label{DP1} Suppose that for a given $Q\subset \Z^d$ the point set $\{\xi^\nu\}_{\nu=1}^m$ has the following properties. There exists an even function $\cV_Q$ such that
for any $f\in \Tr(Q)$ we have
\be\label{D1}
f(\bx) = (2\pi)^{-d} \int_{\T^d} f(\by) \cV_Q(\bx-\by)d\by  = \frac{1}{m}\sum_{\nu=1}^m f(\xi^\nu)\cV_Q(\bx-\xi^\nu)
\ee
and for a given $1\le p\le\infty$ we have for all $\ba=(a_1,\dots,a_m)$ and for $q=p$ and $q=p'$, where $p'$ is dual to $p$ ($1/p+1/p'=1$)
\be\label{D2}
\left\|\frac{1}{m}\sum_{\nu=1}^m a_\nu\cV_Q(\bx-\xi^\nu)\right\|_q \le C(d,p)\|\ba\|_{q,m}.
\ee
Then, the following sampling discretization inequalities hold for all $f\in\Tr(Q)$
\be\label{D3}
C_1(d,p)\|f\|_p^p \le \frac{1}{m}\sum_{\nu=1}^m |f(\xi^\nu)|^p \le C_2(d,p)\|f\|_p^p.
\ee
\end{Proposition}
\begin{proof} The left inequality in (\ref{D3}) with $C_1(d,p)=C(d,p)^{-p}$ directly follows from (\ref{D2}) and (\ref{D1}). We prove the right inequality in (\ref{D3}). Using (\ref{D1}) and (\ref{D2}), we obtain ($\varepsilon_\nu:= \sign f(\xi^\nu)$)
\begin{align*}
\frac{1}{m}\sum_{\nu=1}^m |f(\xi^\nu)|^p & =
\frac{1}{m}\sum_{\nu=1}^{m} f(\xi^\nu)\varepsilon_\nu \bigl| f(\xi^\nu) \bigr|^{p-1} =\\
&=(2\pi)^{-d}\int_{\T^d}f(\bx)\frac{1}{m}\sum_{\nu=1}^{m}\varepsilon_\nu
\bigl|f(\xi^\nu)\bigr|^{p-1}\mathcal V_Q(\bx-\xi^\nu)d\bx\le\\
&\le \|f \|_p\left \|\frac{1}{m}\sum_{\nu=1}^{m}\varepsilon_\nu \bigl| f(\xi^\nu)
\bigr|^{p-1}\mathcal V_Q (\bx - \xi^\nu) \right\|_{p'}.
\end{align*}
Using (\ref{D2}), we see that the last term is
$$
\le C(d,p) \| f \|_p \left(\frac{1}{m}\sum_{\nu=1}^m  \bigl| f(\xi^\nu)
\bigr|^{p}\right)^{(p-1)/p} ,
$$
which implies the required inequality with $C_2(d,p) = C(d,p)^p$.
\end{proof}

The above discussion shows that instead of the universal discretization in a style of Theorem \ref{IT1} the universal simultaneous (of the $L_2$ and $L_p$ norms) sampling discretization problem can be considered. We give a precise formulation of this problem in a general setting. We begin with the known settings.

{\bf Sampling discretization.} Let $\Omega$ be a compact subset of $\R^d$ with the probability measure $\mu$. We say that a linear subspace $X_N$ (index $N$ here, usually, stands for the dimension of $X_N$) of $L_q(\Omega,\mu)$, $1\le q < \infty$, admits the sampling discretization  with parameters $m\in \N$ and $q$ and positive constants $C_1\le C_2$ if there exist a set $\{\xi^\nu\}_{\nu=1}^m$ 
 such that for any $f\in X_N$ we have
$$
C_1\|f\|_q^q \le \frac{1}{m} \sum_{\nu=1}^m |f(\xi^\nu)|^q \le C_2\|f\|_q^q.
$$
In the case $q=\infty$ we define $L_\infty$ as the space of continuous functions on $\Omega$  and ask for
$$
C_1\|f\|_\infty \le \max_{1\le \nu\le m} |f(\xi^\nu)| \le  \|f\|_\infty.
$$

{\bf Universal sampling discretization.} This problem is about finding (proving existence) of 
a set of points, which is good in the sense of the above sampling discretization 
for a collection of linear subspaces (see \cite{VT160}). We formulate it in an explicit form. Let $\cX_\bN:= \{X_{N_j}^j\}_{j=1}^k$, $\bN=(N_1,\dots,N_k)$, be a collection of linear subspaces $X_{N_j}^j$ of the $L_q(\Omega,\mu)$, $1\le q \le \infty$. We say that a set $\{\xi^\nu\}_{\nu=1}^m$ provides {\it universal sampling discretization} for the collection $\cX_\bN$ if, in the case $1\le q<\infty$, there are two positive constants $C_i=C_i(\Omega,q)$, $i=1,2$, such that for each $j\in [1,k]$ and any $f\in X_{N_j}^j$ we have
$$
C_1 \|f\|_q^q \le \frac{1}{m} \sum_{\nu=1}^m |f(\xi^\nu)|^q \le C_2 \|f\|_q^q.
$$
In the case $q=\infty$  for each $j\in [1,k]$ and any $f\in X_N^j$ we have
$$
C_1 \|f\|_\infty \le \max_{1\le\nu\le m} |f(\xi^\nu)| \le  \|f\|_\infty.
$$

{\bf Universal simultaneous sampling discretization.} Let $1\le p_1<p_2\le\infty$ be given.  Let $\cX_\bN:= \{X_{N_j}^j\}_{j=1}^k$, $\bN=(N_1,\dots,N_k)$, be a collection of linear subspaces $X_{N_j}^j$ of the $L_{p_2}(\Omega,\mu)$. We say that a set $\{\xi^\nu\}_{\nu=1}^m$ provides {\it universal simultaneous sampling discretization} for the collection $\cX_\bN$ if, in the case $1< p_2<\infty$, there are four positive constants $C_{i,r}(\Omega,p_r)$, $i=1,2$, $r=1,2$, such that for each $j\in [1,k]$ and any $f\in X_{N_j}^j$ we have
\be\label{D4}
C_{1,r}(\Omega,p_r)\|f\|_{p_r}^{p_r} \le \frac{1}{m} \sum_{\nu=1}^m |f(\xi^\nu)|^{p_r} \le C_{2,r}(\Omega,p_r)\|f\|_{p_r}^{p_r}.
\ee
In the case $p_2=\infty$  for each $j\in [1,k]$ and any $f\in X_N^j$ we have (\ref{D4}) for $r=1$ and for $r=2$ we have
$$
C_{1,2}(\Omega)\|f\|_\infty \le \max_{1\le\nu\le m} |f(\xi^\nu)| \le  \|f\|_\infty.
$$

Note that the case of exact discretization corresponds to the case $C_{1,r}=C_{2,r}$. 

We now discuss universal discretization for a specific collection of subspaces. In \cite{VT160} we studied the universal sampling discretization  for the collection of subspaces of trigonometric polynomials with frequencies from parallelepipeds (rectangles). We now formulate the corresponding result from \cite{VT160}. For $\bs=(s_1,\dots,s_d)\in\Z^d_+$
define
$$
R(2^\bs) := \{\bk =(k_1,\dots,k_d)\in \Z^d :   |k_i| < 2^{s_i}, \quad i=1,\dots,d\}.
$$
  Consider the collection $\cC(n,d):= \{\Tr(R(2^\bs)): \|\bs\|_1=n\}$. 

We proved in \cite{VT160} the following result.  
\begin{Theorem}[\cite{VT160}]\label{DT1}   For every $1\le q\le\infty$ there exists a positive constant $C(d,q)$, which depends only on $d$ and $q$, such that for any $n\in \N$ there is a set $\{\xi^\nu\}_{\nu=1}^m\subset \T^d$, with $m\le C(d,q)2^n$ that provides universal discretization of the $L_q$ norm  for the collection $\cC(n,d)$.
\end{Theorem}
Theorem \ref{DT1} basically solves the universal discretization problem for the collection $\cC(n,d)$. It provides the upper bound  $m\le C(d,q)2^n$ with $2^n$ being of the order of the dimension of each $\Tr(R(2^\bs))$ from the collection $\cC(n,d)$.
 Obviously, the lower bound for the cardinality of a set, providing the sampling discretization for $\Tr(R(2^\bs))$ with $\|\bs\|_1=n$, is $\ge C(d)2^n$. 

The proof of Theorem \ref{DT1} from \cite{VT160} is based on special nets, which we define below.
\begin{Definition}\label{DD1} A $(t,r,d)$-net (in base $2$) is a set $T$ of $2^r$ points in 
$[0,1)^d$ such that each dyadic box $[(a_1-1)2^{-s_1},a_12^{-s_1})\times\cdots\times[(a_d-1)2^{-s_d},a_d2^{-s_d})$, $1\le a_j\le 2^{s_j}$, $j=1,\dots,d$, of volume $2^{t-r}$ contains exactly $2^t$ points of $T$.
\end{Definition}
A construction (which is a very nontrivial construction) of such nets for all $d$ and $t\ge Cd$, where $C$ is a positive absolute constant, $r\ge t$  is given in \cite{NX}. 

It was proved in \cite{VT161} that in the case $d=2$ a very simple point sets, namely the Fibonacci point sets, provide Theorem \ref{DT1}. Here is the corresponding result.

\begin{Theorem}[\cite{VT161}]\label{DT2} The Fibonacci point set $\cF_n$ provides the universal discretization in $L_q$, $1\le q\le\infty$, for the collection $\cC(r,2)$ with $r$ satisfying 
the condition $9\cdot 2^r \le \gamma b_n$.
\end{Theorem}

In this paper it is more convenient for us to consider instead of the collection $\cC(n,d)$ of dyadic rectangles a collection of all rectangles $\cC'(N,d)$.

The following variant of  Theorem \ref{DT2} follows from its proof  in \cite{VT161}.

\begin{Theorem}\label{DT3} The Fibonacci point set $\cF_n$ provides the universal discretization in $L_q$, $1\le q\le\infty$, for the collection $\cC'(N,2)$ with $N$ satisfying 
the condition $9 N \le \gamma b_n$.
\end{Theorem}

As we pointed out above our requirement of discretization of the convolution is equivalent to 
the exact sampling discretization of the $L_2$ norm. It is known that the property of the exact sampling discretization of the $L_2$ norm is much stronger than just sampling discretization with some constants $C_1$ and $C_2$. We now explain this in more detail. 

First, we note that 
for all $\bj$ such that $\prod_{i=1}^d j_i \le N$ we have the embedding $R(\bj) \subset \Gamma(N,d)$. Therefore, for the universal exact discretization of the $L_2$ norm for the collection 
$\cC'(N,d)$ it would be sufficient to provide the exact discretization of the $L_2$ norm for one subspace $\Tr(N,d)$. It is know (see \cite{DPTT}, section 3.5) that for that we need at least $m$ of order $N^2$ points. It is also known (see \cite{VT158}, section 4) that there exists a number theoretical construction of $m\le C(d)N^2(\log N)^{2(d-1)}$ points, which provide 
the exact discretization of the $L_2$ norm for the subspace $\Tr(N,d)$. We point out that 
our Theorem \ref{IT1} provides in the case $d=2$ optimal in the sense of order result for 
the universal sampling discretization for the collection $\cC'(N,d)$ simultaneously exact in case of $L_2$ and non-exact in $L_p$. Theorem \ref{CT1} provides in case of general $d$ 
suboptimal results (up to the logarithmic factor). 

Second, in the case of non-exact sampling discretization of the $L_2$ norm the following 
optimal in the sense of order result is known. 

\begin{Theorem}[\cite{VT158}]\label{DT4} There are three positive absolute constants $C_1$, $C_2$, and $C_3$ with the following properties: For any $d\in \N$ and any $Q\subset \Z^d$   there exists a set of  $m \le C_1|Q| $ points $\xi^\nu\in \T^d$, $\nu=1,\dots,m$, such that for any $f\in \Tr(Q)$ 
we have
$$
C_2\|f\|_2^2 \le \frac{1}{m}\sum_{\nu=1}^m |f(\xi^\nu)|^2 \le C_3\|f\|_2^2.
$$
\end{Theorem}

Theorem \ref{DT4} was derived from results of the paper by  S. Nitzan, A. Olevskii, and A. Ulanovskii \cite{NOU}, which in turn is based on the paper of A. Marcus, D.A. Spielman, and N. Srivastava \cite{MSS}. The breakthrough results in sampling discretization of the $L_2$ norm (see \cite{VT158}, \cite{LT}) and in sampling recovery in the $L_2$ norm (see \cite{KU}, \cite{KU2}, \cite{NSU}) are based on results by A.~Marcus, D.A.~Spielman, and
N.~Srivastava from \cite{MSS} (see Corollary 1.5 with $r=2$ there) obtained for solving the Kadison-Singer problem. Also, results from \cite{BSS} play a fundamental role in sampling discretization of the $L_2$ norm. The approach, based 
on \cite{MSS} allows us to obtain optimal (in the sense of order) results for discretization of the $L_2$ norm (see \cite{VT158} and \cite{LT}). For the first time it was done in \cite{VT158} with the help of a lemma from \cite{NOU}. The corresponding lemma from \cite{NOU}  was further generalized in \cite{LT} for proving optimal in the sense of order sampling discretization results. A version of the corresponding lemma from \cite{LT} was used in \cite{NSU} for the sampling recovery.  The first application of the results from \cite{BSS} in the sampling discretization of the $L_2$ norm was done in \cite{VT159}.
The reader can find a detailed discussion of these results in \cite{KKLT}, Section 2.6.

\subsection{Representation and discretization} 
We now give a general formulation of variants of the setting that was studied in this paper.

{\bf Sampling shift representation  (SSR).}  We say that a linear subspace $X_N$   of $L_\infty(\Omega)$,  admits the SSR  with parameter $m\in \N$ if there exist a function $\ff \in L_\infty(\Omega)$ and a set of points $\{\xi^\nu\}_{\nu=1}^m$
 such that for any $f\in X_N$ we have
$$
f(x) = \frac{1}{m}\sum_{\nu=1}^m f(\xi^\nu)\ff(x-\xi^\nu).
$$
We also say that the above set of points $\{\xi^\nu\}_{\nu=1}^m$ provides SSR
for $X_N$. 

{\bf Universal sampling shift representation   (USSR).} We say that a set of points $\{\xi^\nu\}_{\nu=1}^m$ provides USSR for the collection $\cX_\bN:= \{X_{N_j}^j\}_{j=1}^k$, $\bN=(N_1,\dots,N_k)$, if  for each $j\in [1,k]$  there exists a function $\ff^j\in L_\infty(\Omega)$ such that for any $f\in X_{N_j}^j$ we have
$$
f(x) = \frac{1}{m}\sum_{\nu=1}^m f(\xi^\nu)\ff^j(x-\xi^\nu).
$$


{\bf Shift $p$-representation  (SpR).}  We say that a linear subspace $X_N$   of $L_p(\Omega)$,  admits the SpR  with parameter $m\in \N$ if there exist a function $\ff \in L_p(\Omega)$ and a set of points $\{\xi^\nu\}_{\nu=1}^m$
 such that for any $f\in X_N$ there is a vector $\ba(f)=(a_1(f),\dots,a_m(f))\in \R^m$, which gives a representation  
$$
f(x) = \frac{1}{m}\sum_{\nu=1}^m a_\nu(f)\ff(x-\xi^\nu)
$$
and satisfies the bound
$$
\|f\|_p \le C(\Omega,p)\|\ba\|_{p,m}.
$$
We also say that the above set of points $\{\xi^\nu\}_{\nu=1}^m$ provides SpR
for $X_N$. 
 
{\bf Universal shift $p$-representation  (USpR).} We say that a set of points $\{\xi^\nu\}_{\nu=1}^m$ provides USpR for the collection $\cX_{\bN}$ if this set of points provides SpR for each $X_{N_j}^j$ from the collection $\cX_{\bN}$. 

{\bf Universal sampling shift $p$-representation  (USSpR).} We say that a set of points $\{\xi^\nu\}_{\nu=1}^m$ provides USSpR for the collection $\cX_{\bN}$ if this set of points provides SpR for each $X_{N_j}^j$ from the collection $\cX_{\bN}$ with $\ba(f)= (f(\xi^1),\dots,f(\xi^m))$. 

For instance, a version of Theorem \ref{IT1} gives the following result.
\begin{Theorem}\label{DT5} The Fibonacci point set $\cF_n$ provides the USSpR, $1\le p\le\infty$, for the collection $\cC'(N,2)$ with $N$ satisfying 
the condition $9 N \le \gamma b_n$.
\end{Theorem}

{\bf Open problem 1.} Let $d\in \N$, $d\ge 3$. Is there a constant $C(d)$ such that there exists a point set $\{\xi^\nu\}_{\nu=1}^m$ with $m\le C(d)N$, which provides the USSR for the collection $\cC'(N,d)$?

 {\bf Open problem 2.} Let $d\in \N$, $d\ge 3$, and $1\le p\le \infty$ be given. Is there a constant $C(d,p)$ such that there exists a point set $\{\xi^\nu\}_{\nu=1}^m$ with $m\le C(d,p)N$, which provides the USSpR for the collection $\cC'(N,d)$?
 
 {\bf Open problem 3.} Let $d\in \N$, $d\ge 3$. Is there a constant $C(d)$ such that there exists a point set $\{\xi^\nu\}_{\nu=1}^m$ with $m\le C(d)N$, which provides the USSpR for all $1\le p\le \infty$ for the collection $\cC'(N,d)$?

{\bf Acknowledgements.}  
The work was supported by the Russian Federation Government Grant N{\textsuperscript{\underline{o}}}14.W03.31.0031.

\end{document}